\documentclass[10pt,conference]{IEEEtran}
\IEEEoverridecommandlockouts
\usepackage{needspace, amsthm}
\makeatletter
\def\ps@headings{%
\def\@oddhead{\mbox{}\scriptsize\rightmark \hfil \thepage}%
\def\@evenhead{\scriptsize\thepage \hfil \leftmark\mbox{}}%
\def\@oddfoot{}%
\def\@evenfoot{}}
\makeatother
\usepackage{amsfonts, amsmath}
\usepackage{url}
\usepackage{subfigure}
\usepackage{wrapfig}
\newlength{\smpagewidth}
\newlength{\smpageheight}

\newcommand{\setleftmargin}[1]{
	\addtolength{\textwidth}{\oddsidemargin}
	\addtolength{\textwidth}{1in}
	\addtolength{\textwidth}{-#1}
	\setlength{\oddsidemargin}{-1in}
	\addtolength{\oddsidemargin}{#1}
	\setlength{\evensidemargin}{\oddsidemargin}
}
\newcommand{\setrightmargin}[1]{
	\setlength{\textwidth}{\smpagewidth}
	\addtolength{\textwidth}{-\oddsidemargin}
	\addtolength{\textwidth}{-1in}
	\addtolength{\textwidth}{-#1}
}
\newcommand{\settopmargin}[1]{
	\addtolength{\textheight}{\topmargin}
	\addtolength{\textheight}{1in}
	\addtolength{\textheight}{\headheight}
	\addtolength{\textheight}{\headsep}
	\addtolength{\textheight}{-#1}
	\setlength{\topmargin}{-1in}
	\addtolength{\topmargin}{-\headheight}
	\addtolength{\topmargin}{-\headsep}
	\addtolength{\topmargin}{#1}
}
\newcommand{\setbottommargin}[1]{
	\setlength{\textheight}{\smpageheight}
	\addtolength{\textheight}{-\topmargin}
	\addtolength{\textheight}{-1in}
	\addtolength{\textheight}{-\footskip}
	\addtolength{\textheight}{-#1}
}

\usepackage{amsfonts}
\usepackage{amssymb}
\usepackage{acronym}
\usepackage{color}

\usepackage{graphicx}
\usepackage{color}

\newcommand{\field}[1]{\mathbb{#1}}

\def\l{\lambda}

\def\e{\epsilon}
\def\g{\gamma}

\def\a{\alpha}

\def\rn{\mathbb{R}^n}

\def\re{\mathbb{R}}

\newtheorem{assumption}{Assumption}

\newtheorem{lemma}{Lemma}

\newtheorem{proposition}{Proposition}

\renewenvironment{proof}[1]{\needspace{1\baselineskip}\noindent {\bf Proof #1:} } {\hfill $\blacksquare$ \medskip}{}

\setleftmargin{.72in}
\setrightmargin{.72in}
\settopmargin{.75in}
\setbottommargin{.77in}

\begin{document}

%
%

\title{\vspace{.1in}
\huge Accelerated Dual Descent for Network Optimization\thanks{This research is supported by Army Research Lab MAST
Collaborative Technology Alliance, AFOSR complex networks program, ARO P-57920-NS, NSF CAREER CCF-0952867, and NSF CCF-1017454, ONR MURI N000140810747 and NSF-ECS-0347285.}}

\author{Michael Zargham$^{\dagger}$, Alejandro Ribeiro$^{\dagger}$, Asuman Ozdaglar$^\ddagger$, Ali Jadbabaie$^{\dagger}$ \thanks{$^{\dagger}$Michael Zargham, Alejandro Ribeiro and Ali Jadbabaie are with the Department of Electrical and Systems Engineering, University of Pennsylvania.}
\thanks{$^\ddagger$ Asuman Ozdaglar is with the Department of Electrical Engineering and Computer
Science, Massachusetts Institute of Technology. }}

\maketitle

\thispagestyle{empty}

\begin{abstract}
Dual descent methods are commonly used to solve network optimization problems because their implementation can be distributed through the network. However, their convergence rates are typically very slow. This paper introduces a family of dual descent algorithms that use approximate Newton directions to accelerate the convergence rate of conventional dual descent. These approximate directions can be computed using local information exchanges thereby retaining the benefits of distributed implementations. The approximate Newton directions are obtained through matrix splitting techniques and sparse Taylor approximations of the inverse Hessian. We show that, similarly to conventional Newton methods, the proposed algorithm exhibits superlinear convergence within a neighborhood of the optimal value. Numerical analysis corroborates that convergence times are between one to two orders of magnitude faster than existing distributed optimization methods. A connection with recent developments that use consensus iterations to compute approximate Newton directions is also presented.
\end{abstract}

\section{Introduction}

Conventional approaches to network optimization are based on subgradient descent in either the primal or dual domain; see, e.g., \cite{mung, kelly, low, Srikant}. For many classes of problems, subgradient descent algorithms yield iterations that can be implemented through distributed updates based on local information exchanges. However, practical applicability of the resulting algorithms is limited by exceedingly slow convergence rates. To overcome this limitation second order Newton methods could be used, but this would require the computation of Newton steps which cannot be accomplished through local information exchanges. This issue is solved in this paper through the introduction of a family of approximations to the Newton step.

The particular problem we consider is the network flow problem. Network connectivity is modeled as a directed graph and the goal of the network is to support a single information flow specified by incoming rates at an arbitrary number of sources and outgoing rates at an arbitrary number of sinks. Each edge of the network is associated with a concave function that determines the cost of traversing that edge as a function of flow units transmitted across the link. Our objective is to find the optimal flows over all links. Optimal flows can be found by solving a concave optimization problem with linear equality constraints (Section II). In particular, the use of subgradient descent in the dual domain allows the development of a distributed iterative algorithm. In this distributed implementation nodes keep track of variables associated with their outgoing edges and undertake updates based on their local variables and variables available at adjacent nodes (Section II-A). Distributed implementation is appealing because it avoids the cost and fragility of collecting all information at a centralized location. However, due to low convergence rates of subgradient descent algorithms, the number of iterations necessary to find optimal flows is typically very large \cite{averagepaper, CISS-rate}. The natural alternative is the use of second order Newton's methods, but they cannot be implemented in a distributed manner (Section II-B).

Indeed, implementation of Newton's method necessitates computation of the inverse of the dual Hessian and a distributed implementation would require each node to have access to a corresponding row. It is not difficult to see that the dual Hessian is in fact a weighted version of the network's Laplacian and that as a consequence its rows could be locally computed through information exchanges with neighboring nodes. Its inversion, however, requires global information. Our insight is to consider a Taylor's expansion of the inverse Hessian, which, being a polynomial with the Hessian matrix as variable, can be implemented through local information exchanges. More precisely, considering only the zeroth order term in the Taylor's expansion yields an approximation to the Hessian inverse based on local information only -- which, incidentally, coincides with the method of Hessian diagonal inverses proposed in \cite{lowdiag}. The first order approximation necessitates information available at neighboring nodes and in general, the $N$th order approximation necessitates information from nodes located $N$ hops away (Section III). The resultant family of algorithms, denoted ADD-N permits a tradeoff between accurate Hessian approximation and communication cost. Despite the fact that the proposed distributed algorithms rely on approximate Newton directions, we show that they exhibit local quadratic convergence as their centralized counterparts (Section IV). An approximate backtracking line search is added to the basic algorithm to ensure global convergence (Section IV-B).

Newton-type methods for distributed network optimization have been recently proposed in \cite{lowdiag, cdc09, wei}. While specifics differ, these papers rely on consensus iterations to compute approximate Newton directions. Quite surprisingly, it is possible to show that the methods in \cite{lowdiag, cdc09, wei} and the ADD algorithm proposed here are equivalent under some conditions (Section V). Numerical experiments study the communication cost of ADD relative to \cite{lowdiag, cdc09, wei} and to conventional subgradient descent. ADD reduces this cost by one order of magnitude with respect to \cite{lowdiag, cdc09, wei} and by two orders of magnitude with respect to subgradient descent (Section VI).

\section{The Network Optimization Problem}\label{mincost-Newton}

Consider a network represented by a directed graph ${\cal G}=({\cal N},{\cal E})$ with node set ${\cal N}=\{1,\ldots,n\}$, and edge set ${\cal E} = \{1,\ldots,E\}$.  The $i$th component of vector $x$ is denoted as $x^{i}$. The notation $x\ge 0$ means that all components $x^i\ge 0$. The network is deployed to support a single information flow specified by incoming rates $b^{i}>0$ at source nodes and outgoing rates $b^{i}<0$ at sink nodes. Rate requirements are collected in a vector $b$, which to ensure problem feasibility has to satisfy $\sum_{i=1}^{n}b^{i}=1$. Our goal is to determine a flow vector $x=[x^e]_{e\in {\cal E}}$, with $x^e$ denoting the amount of flow on edge $e=(i,j)$. We require the some additional notation; the transpose of vector $x$ is $x'$, the inner product of $x$ and $y$ is $x'y$, and the Euclidean norm of $x$ is $\|x\|:=\sqrt{x'x}$.

Flow conservation implies that it must be $Ax=b$, with $A$ the $n\times E$ node-edge incidence matrix defined as

\[ [A]_{ij} = \left\{
\begin{array}{ll}
 1 & \hbox{if edge $j$ leaves node $i$} , \\
-1 & \hbox{if edge $j$ enters node $i$},  \\
 0 & \hbox{otherwise.}
\end{array}\right.\]
The element in the $i$th row and $j$th column of a matrix $A$ is written as $[A]_{ij}$. The transpose of $A$ is denoted as $A'$.  We define the reward as the negative of scalar cost function $\phi_e(x^{e})$ denoting the cost of $x^e$ units of flow traversing edge $e$. We assume that the cost functions $\phi_e$ are strictly convex and twice continuously differentiable. The max reward network optimization problem is then defined as

\begin{equation}
	\hbox{maximize } -f(x)=\sum_{e=1}^E -\phi_e(x^e),\quad \hbox{subject to: } Ax=b.\label{optnet}\\
\end{equation}
Our goal is to investigate Newton-type iterative distributed methods for solving the optimization problem in (\ref{optnet}). Before doing that, let us discuss the workhorse distributed solution based on dual subgradient descent (Section II-A) and the conventional centralized Newton's method (Section II-B).

\subsection{Dual Subgradient Method}

Dual subgradient descent solves (\ref{optnet}) by descending in the dual domain. Start then by defining the Lagrangian function of problem (\ref{optnet}) as ${\cal L} (x,\l) = -\sum_{e=1}^E \phi_e(x^e) +\l'(Ax-b)$ and the dual function $q(\l)$ as

\begin{eqnarray} \label{eqn_separable_dual}
	q(\l)  \!\!&=&\!\! \sup_{x\in \re^E} {\cal L} (x,\l)=\sup_{x\in \re^E}
 		               \left(-\sum_{e=1}^E \phi_e(x^e) +\l'Ax\right) - \l'b \nonumber\\
		       &=&\!\! \sum_{e=1}^E \sup_{x^e \in \re} \Big(-\phi_e(x^e) + (\l'A)^e x^e\Big) - \l'b,
\end{eqnarray}
where in the last equality we wrote $\l'Ax = \sum_{e=1}^{E}(\l'A)^e x^e$ and exchanged the order of the sum and supremum operators.

It can be seen from (\ref{eqn_separable_dual}) that the evaluation of the dual function $q(\l)$ decomposes into $E$ one-dimensional optimization problems that appear in the sum. We assume that each of these problems has an optimal solution, which is unique because of the strict convexity of the functions $\phi_e$. Denote this unique solution as $x^e(\l)$ and use the first order optimality conditions for these problems in order to write
\begin{equation}
	x^e(\l) = (\phi_e')^{-1} (\l^i-\l^j),\label{primal-sol}
\end{equation}
where $i\in{\cal N}$ and $j\in{\cal N}$ respectively denote the source and destination nodes of edge $e=(i,j)$. As per (\ref{primal-sol}) the evaluation of $x^e(\l)$ for each node $e$ is based on local information
about the edge cost function $\phi^e$ and the dual variables of the incident nodes $i$ and $j$.

The dual problem of (\ref{optnet}) is defined as $\min_{\l\in \re^n} q(\l)$. The dual function is convex, because all dual functions of minimization problems are, and differentiable, because the $\phi_e$ functions are strictly convex. Therefore, the dual problem can be solved using gradient descent.  Consider an iteration index $k$, an arbitrary initial vector $\l_0$ and define iterates $\l_k$ generated by the following:

\begin{equation}
	\l_{k+1} = \l_k - \a_k g_k\qquad \hbox{for all }k\ge 0, \label{sgit}
\end{equation}
where $g_k = g(\l_{k})=\nabla q(\l_{k})$ denotes the gradient of the dual function $q(\l)$ at $\l=\l_k$. A first important observation here is that we can compute the gradient as $g_k = A x(\l_k)-b$ with the vector $x(\l_k)$ having components $x^e(\l_k)$ as determined by (\ref{primal-sol}) with $\l=\l_k$,\cite[Section 6.4]{nlp}. Differentiability of $g(\lambda)$ follows from strict convexity of (\ref{optnet}).  A second important observation is that because of the sparsity pattern of the node-edge incidence matrix $A$ the $i$th element $g_k^{i}$ of the gradient $g_{k}$ can be computed as

\begin{equation}\label{eqn_dual_update_distributed}
	g_k^{i} = \sum_{e =(i,j)} x^e(\l_{k}) -  \sum_{e =(j,i)} x^e(\l_{k})  - b_{i}
\end{equation}
The algorithm in (\ref{sgit}) lends itself to distributed implementation. Each node $i$ maintains information about its dual iterates $\l^i_{k}$ and primal iterates $x^e(\l_{k})$ of outgoing edges $e=(i,j)$.
Gradient components $g_k^{i}$ are evaluated as per (\ref{eqn_dual_update_distributed}) using local primal iterates $x^e(\l_{k})$ for $e =(i,j)$ and primal iterates of neighboring nodes $x^e(\l_{k})$ for $e=(j,i)$. Dual variables are then updated as per (\ref{sgit}). Having updated the dual iterates, we proceed to update primal variables as per (\ref{primal-sol}). This update necessitates local multipliers $\l^{i}_{k}$ and neighboring multipliers $\l_{k}^{j}$.

Distributed implementation is appealing because it avoids the cost and fragility of collecting all information at a centralized location. However, practical applicability of gradient descent algorithms is hindered by slow convergence rates; see e.g., \cite{averagepaper,CISS-rate}. This motivates consideration of Newton's method which we describe next.

\subsection{Newton's Method for Dual Descent}

Newton's Method is a descent algorithm along a scaled version of the gradient. In lieu of (\ref{sgit}) iterates are given by
\begin{equation}\l_{k+1} = \l_k + \a_k d_k\qquad \hbox{for all }k\ge 0, \label{newtup}\end{equation}
where $d_k$ is the Newton direction at iteration $k$ and $\alpha_{k}$ is a properly selected step size. The Newton direction, $d_k$ satisfies
\begin{equation}
	H_k d_k = -g_k, \label{newt}
\end{equation}
where $H_k=H(\l_{k})=\nabla^{2}q(\l_{k})$ is the Hessian of the dual function at the current iterate and $g_k=g_{k}(\l_{k})$ is, we recall, the corresponding gradient.

To obtain an expression for the dual Hessian, consider given dual $\l_k$ and primal $x_k=x(\lambda_k)$ variables, and consider the second order approximation of the primal objective centered at the current primal iterates $x_{k}$,
\begin{equation} \label{eqn_primal_quadratic_approximation}
	\hat f(y) = f(x_k)+ \nabla f(x_k)'(y-x_k) + \frac{1}{2} (y-x_k)'\nabla^2 f(x_k) (y-x_k)\end{equation}
The primal optimization problem in (\ref{optnet}) is now replaced by the maximization of the approximating function $-\hat f(y)$ in (\ref{eqn_primal_quadratic_approximation}) subject to the constraint $Ay=b$. This approximated problem is a quadratic program whose dual is the (also quadratic) program
\begin{equation}\label{eqn_dual_quadratic_approximation}
	\min_{\l\in \re^n} g(\lambda_k) =
	\min_{\l\in \re^n}	  \frac{1}{2}\lambda_k' A\nabla^2 f(x_k)^{-1} A' \lambda_k + p'\lambda_k + r.
\end{equation}
The vector $p$ and the constant $r$ can be expressed in closed form as functions of $\nabla f(x_k)$ and $\nabla^2 f(x_k)$, but they are irrelevant for the discussion here. The important consequence of (\ref{eqn_dual_quadratic_approximation}) is that the dual Hessian is given by
\begin{equation}
	H_k=Q=A\left(- \nabla^2f(x_k)^{-1}\right) A'.
\end{equation}
From the definition of $f(x)$ in (\ref{optnet}) it follows that the primal Hessian -$\nabla^2 f(x_k)$ is a diagonal matrix, which is negative definite by strict convexity of $f(x)$. Therefore, its inverse exists and can be computed locally. Further observe that the dual Hessian, being the product of the incidence matrix times a positive definite diagonal matrix times the incidence matrix transpose, is a weighted version of the network graph's Laplacian. As a particular consequence it follows that $\mathbf{1}$ is an eigenvector of $H_{k}$ associated with eigenvalue $0$ and that $H_{k}$ is invertible on the the subspace $\mathbf{1}^\perp$. Since the gradient $g_k$ lies in $\mathbf{1}^\perp$ we can find the Newton step as $d_k = -H_k^\dagger g_k$. However, computation of the pseudoinverse $H_k^\dagger$ requires global information. We are therefore interested in approximations of the Newton direction requiring local information only.

\section{Approximate Newton's Method}
To define an approximate Newton direction, i.e., one for which (\ref{newt}) is approximately true, we will consider a finite number of terms of a suitable Taylor's expansion representation of the Newton direction.
At iteration $k$, split the Hessian into diagonal elements $D_{k}$ and off diagonal elements $B_{k}$ and write $H_k=D_k-B_k$. Further rewrite the Hessian as $H_k=D_k^{-\frac{1}{2}}\left(I-D_k^{-\frac{1}{2}}B_kD_k^{-\frac{1}{2}}\right)D_k^{-\frac{1}{2}}$, which implies that the Hessian pseudo-inverse is given by $H_k^{-\dagger}=D_k^{-\frac{1}{2}}\left(I-D_k^{-\frac{1}{2}}B_kD_k^{-\frac{1}{2}}\right)^{-\dagger}D_k^{-\frac{1}{2}}$. Notice now that for the central term of this product we can use the Taylor's expansion identity $(I-X)^{\dagger}v = \left(\sum_{i=0}^\infty X^i\right)v$, which is valid for any vector $v$ orthogonal to the eigenvectors of $X$ associated with eigenvalue 1. Since $g_k$ is orthogonal to ${\bf 1}$, it follows

\begin{eqnarray*}
    d_k = -H_k^\dagger g_k
    		= -\sum_{i=0}^\infty D_k^{-\frac{1}{2}}\left(D_k^{-\frac{1}{2}} B_k D_k^{-\frac{1}{2}}\right)^i D_k^{-\frac{1}{2}}g_k .
\end{eqnarray*}
The Newton direction is now represented as an infinite sum we can define a family of approximations characterized by truncations of this sum,
\begin{equation} \label{dir}
	d_k^{(N)} \!= -\! \sum_{i=0}^N D_k^{-\frac{1}{2}}\!\left(D_k^{-\frac{1}{2}} B_k D_k^{-\frac{1}{2}}\right)^i\!D_k^{-\frac{1}{2}}\!g_k
			   := - \bar H_{k}^{(N)} g_k,
\end{equation}
where we have defined the approximate Hessian pseudo inverse $\bar H_{k}^{(N)}:=\sum_{i=0}^N D_k^{-\frac{1}{2}}\left(D_k^{-\frac{1}{2}} B_k D_k^{-\frac{1}{2}}\right)^i D_k^{-\frac{1}{2}}$. The approximate Newton algorithm is obtained by replacing the Newton step $d_{k}$ in (\ref{newtup}) by its approximations $d_k^{(N)}=- \bar H_{k}^{(N)} g_k$. The resultant algorithm is characterized by the iteration
\begin{equation}
	\l_{k+1} = \l_k - \a_k \bar H_{k}^{(N)} g_k \label{update}.
\end{equation}
While not obvious, the choice of $N$ in (\ref{dir}) dictates how much information node $i$ needs from the network in order to compute the $i$th element of the approximate Newton direction $d_k^{(N)}$ -- recall that node $i$ is associated with dual variable $\l_{k}^{i}$.

For the zeroth order approximation $d_k^{(0)}$ only the first term of the sum in (\ref{dir}) is considered and it therefore suffices to have access to the information in $D_{k}$ to compute the approximate Newton step. Notice that the approximation in this case reduces to $d_k^{(0)}= D_k^{-1} g_k$ implying that we approximate $H_{k}^{-1}$ by the inverse diagonals which coincides with the method in \cite{lowdiag}.

The first order approximation $d_k^{(1)}$ uses the first two terms of the sum in (\ref{dir}) yielding $d_k^{(1)} = \left(D_k^{-1}+ D_k^{-1} B_k D_k^{-1}\right)g_{k}$. The key observation here is that the sparsity pattern of $B_{k}$, and as a consequence the sparsity pattern of $D_k^{-1} B_k D_k^{-1}$, is that of the graph Laplacian, which means that $[D_k^{-1} B_k D_k^{-1}]_{ij}\neq0$ if and only if $i$ and $j$ correspond to an edge in the graph, i.e, $(i,j)\in E$. As a consequence, to compute the $i$th element of $d_k^{(1)}$ node $i$ needs to collect information that is either locally available or available at nodes that share an edge with $i$.

For the second order approximation $d_k^{(2)}$ we add the term $\left(D_k^{-1} B_k\right)^2 D_k^{-1}$ to the approximation $d_k^{(1)}$. The sparsity pattern of $\left(D_k^{-1} B_k\right)^2 D_k^{-1}$ is that of $B_k^{2}$, which is easy to realize has nonzero entries matching the 2-hop neighborhoods of each node. Therefore, to compute the $i$th element of $d_k^{(2)}$ node $i$ requires access to information from neighboring nodes and from neighbors of these neighbors. In general, the $N$th order approximation adds a term of the form  $\left(D_k^{-1} B_k\right)^N D_k^{-1}$ to the $N-1$st order approximation. The sparsity pattern of this term is that of $B_k^{N}$, which coincides with the $N$-hop neighborhood, and computation of the local elements of the Newton step necessitates information from $N$ hops away.

We thus interpret (\ref{dir}) as a family of approximations indexed by $N$ that yields Hessian approximations requiring information from $N$-hop neighbors in the network. This family of methods offers a trade off between communication cost and precision of the Newton direction. We analyze convergence properties of these methods in the coming sections.

\subsection{Convergence}

A basic guarantee for any iterative optimization algorithm is to show that it eventually approaches a neighborhood of the optimal solution. This is not immediate for ADD as defined by (\ref{update}) because the errors in the $\bar H_{k}^{(N)}$ approximations to $H_{k}^{\dagger}$ may be significant. Notwithstanding, it is possible to prove that the $\bar H_{k}^{(N)}$ approximations are positive definite for all $N$ and from there to conclude that the $\l_{k}$ iterates in (\ref{update}) eventually approach a neighborhood of the optimal $\l^{*}$. This claim is stated and proved in the next proposition.

\begin{proposition} \label{prop_basic_convergence}
Let $\l^{*}$ denote the optimal argument of the dual function $q(\l)$ of the optimization problem in (\ref{optnet}) and consider the ADD-$N$ algorithm characterized by iteration (\ref{update}) with $\bar H_{k}^{(N)}$ as in (\ref{dir}). Assume $\alpha_{k}=\alpha$ for all $k$ and that the network graph is not bipartite. Then, for all sufficiently small $\alpha$,
\begin{equation}
	\lim_{k\to\infty} \l_{k} = \l^{*}
\end{equation}\end{proposition}

\begin{proof}{}
As per the Descent Lemma, to prove convergence of the $\l_{k}$ iterates in (\ref{update}) it suffices to show that the matrix $\bar H_k^{(N)}$ is positive definite for all $k$ \cite[Proposition A.24]{nlp}.

To do so, begin by recalling that the dual Hessian, $H_k$ is a weighted Laplacian and define the normalized Laplacian $L_k=D_k^{-\frac{1}{2}}H_k D_k^{-\frac{1}{2}}$, having unit diagonal elements. Applying the splitting $H_{k}=D_{k}-B_{k}$ it follows $D_k^{-\frac{1}{2}}B_kD_k^{-\frac{1}{2}}= I-L_k$ from which we can write $\bar H_{k}^{(N)}$ as [cf. (\ref{dir})]

\begin{equation}
	\bar H_{k}^{(N)}  = D_k^{-\frac{1}{2}} \left( \sum_{i=0}^N \big(I-L_{k} \big)^i\right)D_k^{-\frac{1}{2}}.
\end{equation}
Focus now on the sum $\sum_{i=0}^N \big(I-L_{k} \big)$. For non-bipartite graphs, normalized Laplacians have eigenvalues in the interval $\left[0,2\right)$, \cite[Lemma 1.7]{fanbook}.  Therefore, it follows that the eigenvalues of $I-L_k$, fall in $\left(-1,1\right]$. Furthermore, the normalized Laplacian $L_k$ has exactly one eigenvector, $\nu_0$ associated with the eigenvalue $0$. Observe that since $L_{k}\nu_{0}=0$, $\nu_{0}$ satisfies
\[\nu_0' \left(\sum_{i=0}^N (I-L_k)^i\right) \nu_0 = (N+1) \nu_0'\nu_0>0.\]
The other eigenvectors of the normalized Laplacian necessarily lie in $\left(-1,1\right)$. Suppose $\mu \in (-1,1)$ is one such eigenvalue of $I-L_k$, associated with eigenvector $\nu$. It follows that $\nu$ is an eigenvector of $\sum_{i=0}^N (I-L_k)^i$, whose associated eigenvalue is $\nu$ is $(1-\mu^{N+1})/(1-\mu)>0$ as follows from the sum of the truncated geometric series. Since the latter value is positive for any $\mu\in (-1,1)$ it follows that $\sum_{i=0}^N (I-L_k)^i$ is positive definite. Further observe that it is also symmetric by definition so we can define $\sum_{i=0}^N (I-L_k)^i = C'C$
where $C$ is square and full rank, \cite[Theorem 7.2.10]{HJ}. We then construct $\bar C = C D_k^{-\frac{1}{2}}$ which is also square and full rank.  This gives us a new symmetric positive definite matrix $\bar C' \bar C=D_k^{-\frac{1}{2}}\sum_{n=0}^N (I-L_k)^n D_k^{-\frac{1}{2}} = \bar H_{k}^{(N)},$ thus completing the proof.
\end{proof}

By continuity of (\ref{primal-sol}), convergence of the dual variable to an error neighborhood implies convergence of the primal variables to an error neighborhood.  Requiring the graph to {\it not} be bipartite is a technical condition to avoid instabilities created by Laplacian eigenvalues $-1$. The restriction is not significant in practice.

\section{Convergence Rate}

The basic guarantee in Proposition \ref{prop_basic_convergence} is not stronger than convergence results for regular gradient descent. Our goal is to show that the approximate Newton method in (\ref{update}) exhibits quadratic convergence in a sense similar to centralized (exact) Newton algorithms. Specifically, we will show that selecting $N$ large enough, it is possible to find a neighborhood of $\l^{*}$ such that if the iteration is started within that neighborhood iterates converge quadratically.

Before introducing this result let us define the Newton approximation error $\e_k$ as
\vspace{-.1in}
\begin{equation}
	\e_k = H_k d_k^{(N)} +g_k. \label{approxNewdir}
\end{equation}
We further introduce the following standard assumptions to bound the rate of change in the Hessian of the dual function.

\begin{assumption} \label{residual}
The Hessian $H(\l)$ of the dual function $q(\l)$ satisfies the following conditions 
\begin{list}{}{\setlength{\itemsep  }{2pt} \setlength{\parsep    }{2pt}
                                           \setlength{\parskip  }{0in} \setlength{\topsep    }{2pt}
                                           \setlength{\partopsep}{0in} \setlength{\leftmargin}{10pt}
                                           \setlength{\labelsep }{10pt} \setlength{\labelwidth}{-0pt}}
\item[({\it Lipschitz dual Hessian})] There exists some constant $L>0$ such that
$\|H(\l)-H(\bar{\l})\| \le L\|\l-\bar{\l}\| \, \forall \l,\bar{\l}\in \rn.$
\item [({\it Strictly convex dual function})] There exists some constant $M>0$ such that $\|H(\l)^{-1}\|\le M \qquad \, \forall \l\in \rn.$
\end{list}\end{assumption}

As is usual in second order optimization methods we use the gradient norm $\| g_{k}\| = \| g(\l_{k})\|$ to measure the progress of the algorithm. The aforementioned quadratic convergence result establishes that for any graph we can always select $N$ large enough so that if an iterate $\l_{k}$ is sufficiently close to the optimal $\l^{*}$, the gradient norm $\| g_{k+m}\|$ of subsequent iterates $\l_{k+m}$ decays like $2^{2^{m}}$. This is formally stated in the following.

\begin{proposition} \label{mainresult}
Consider ADD-$N$ algorithms characterized by the iteration (\ref{update}) with $\bar H_{k}^{(N)}$ as defined in (\ref{dir}). Let Assumption \ref{residual} hold and further assume that the step size is $\alpha_{k+m}=1$ for all $k\geq m$. Let $\e$ be a uniform bound in the norm of the Newton approximation error $\e_{k}$ in (\ref{approxNewdir}) so that $\|\e_k\| \le \e$ for all $k$. Define the constant
\vspace{-.1in}\begin{equation}B=\e+M^2L \e^2.\label{sumerr}\end{equation}
Further assume that at time $k$ it holds $\|g_k\|\le 1/(2M^2 L)$ and that $N$ is chosen large enough to ensure that for some $\delta\in (0,1/2)$, $B+M^2 L B^2 \le \delta/(4 M^2 L)$. Then, for all $m\geq1$,
\begin{equation}\label{quadconv}
	\|g_{k+m}\|\le {1\over 2^{2^m}M^2 L} + B + \frac{\delta}{M^2 L}\, \frac{(2^{2^m-1}-1)}{2^{2^m}}.
\end{equation}
In particular, as $m\to\infty$ it holds
\begin{equation}
	\limsup_{m\to \infty}\|g_{k+m}\|\le B + {\delta \over 2M^2 L}.
\end{equation}\end{proposition}

Proposition \ref{mainresult} has the same structure of local convergence results for Newton's method \cite[Section 9.5]{boydbook}. In particular, quadratic convergence follows from the term $1/\left(2^{2^m}\right)$ in (\ref{quadconv}). The remaining terms in (\ref{quadconv}) are small constants that account for the error in the approximation of the Newton step.

Notice that Proposition \ref{mainresult} assumes that at some point in the algorithm's progression, $\|g_k\|\le 1/(2M^2 L)$. Quadratic convergence is only guaranteed for subsequent iterates $\l_{k+m}$. This is not a drawback of ADD, but a characteristic of all second order descent algorithms. To ensure that some iterate $\l_{k}$ does come close to $\l^{*}$ so that $\|g_k\|\le 1/(2M^2 L)$ we use a distributed adaptation of backtracking line search (Section \ref{sec_backtracking}).

To proceed with the proof of Proposition \ref{mainresult} we need two preliminary results. The first result concerns the bound $\e$ which was required to hold uniformly for all iteration indexes $k$. While it is clear that increasing $N$ reduces $\|\e_{k}\|$, it is not immediate that a uniform bound should exist. The fact that a uniform bound does exists is claimed in the following lemma.

\begin{lemma} \label{boundprop}
Given an arbitrary $\epsilon >0$, there exists an $N$ such that the Newton approximation errors $\e_{k}$ as defined in (\ref{approxNewdir}) have uniformly bounded norms $\|\e_{k}\|\leq\e$ for all iteration indexes $k$.
\end{lemma}

\begin{proof}{}
We begin eliminating the summation from our expression of the Newton error by observing that a telescopic property emerges.
\begin{eqnarray*}
H_kd_k^{(N)} + g_k&=& H_k\left(-\sum_{i=0}^N\left(D_k^{-1}B_k\right)^iD_k^{-1}g_k\right)+g_k\\
&=& \left(I-(D_k-B_k)\sum_{i=0}^N\left(D_k^{-1}B_k\right)^iD_k^{-1}\right)g_k\\
&=& \left(I - \sum_{i=0}^N  (D_k^{-1}B_k)^i-(D_k^{-1}B_k)^{i+1} \right) g_k\\
&=& (B_kD_k^{-1})^{N+1}g_k
\end{eqnarray*}
We introduce the matrix $V\in \field{R}^{n \times n-1}$, made up of $n-1$ orthonormal columns spanning $\mathbf{1}^\perp$.  We observe that $VV'= I_n- \frac{\mathbf{11}'}{n}$, and since $g\in \mathbf{1}^{\perp}$ we have $g=VV'g$.  Our descent occurs in $\mathbf{1}^\perp$ so we restrict our analysis to this subspace.
We have $\Vert V'(B_kD_k^{-1})^{N+1}g_k \Vert =  \Vert V'(B_kD_k^{-1})^{N+1}VV'g_k \Vert \le \Vert V'(B_kD_k^{-1})^{N+1}V\Vert \Vert V'g_k\Vert\le \rho^{N+1} \left(B_kD_k^{-1}\right) \Vert g_k \Vert$ from the triangle inequality and the following definition.  For a matrix X, $\rho(X)$ is the radius of a disc containing all eigenvalues with subunit magnitude.  In this problem $\rho(V'(B_kD_k^{-1})V)$ coincides with the largest eigenvalue modulus and $\rho(B_kD_k^{-1})$ is the second largest eigenvalue modulus.
From \cite{Landau} we have \begin{equation}\rho\left(B_kD_k^{-1}\right)\le 1-\frac{1}{n\Delta(G)(\hbox{diam}(G)+1)b_{\hbox{max}}}\end{equation}
where $\Delta(G)$ is the maximum degree of any node in $G$, $\hbox{diam}(G)$ is the diameter of $G$ and $b_{\hbox{max}}$ is an upper bound on dual off diagonal elements of the dual hessian: $[H_k]_{ij}\le b_{\hbox{max}} \, \forall i\not = j$.  Combining this fact with the assumption that $g_k$ is upper bounded for all $k$, the result follows.
\end{proof}


Another preliminary result necessary for the proof of Proposition \ref{mainresult} is an iterative relationship between the gradient norm $\|g_{k+1}\|$ at iteration $k+1$ and the norm $\|g_k\|$ at iteration $k$. This relationship follows from a multi-dimensional extension of the descent lemma (see \cite{ourbook}) as we explain next.

\begin{lemma}{Let Assumption \ref{residual} hold.
Let $\{\l_k\}$ be a sequence generated by the method
(\ref{newtup}). For any stepsize rule $\a_k$, we have
$\|g_{k+1}\| \le (1-\a_k) \|g_k\| + M^2 L \a_k^2 \|g_k\|^2+\a_k \|\e_k\|+ M^2 L\a_k^2\|\e_k\|^2.$
}\label{basicrel}
\end{lemma}

\begin{proof}{}
We consider two vectors $w\in \rn$ and $z\in$. We let $\xi$ be a scalar parameter and define the function
$y(\xi)= \nabla y(w+\xi z)$. From the chain rule, it follows that $\frac{\partial}{\partial \xi} y(\xi)= H(w+\xi z) z$.
Using the Lipschitz continuity of the
residual function gradient [cf.\ Assumption \ref{residual}(a)], we
obtain:  $g(w+z)-g(w)=y(1)-y(0)=\int_{0}^1 \frac{\partial}{\partial \xi} y(\xi) d\xi=\int_0^1 H(w+\xi z) z d\xi$
\begin{eqnarray*}
&\le&  \left| \int_0^1 (H(w+\xi z)-H(w))zd\xi\right| + \int_0^1 H(w)zd\xi\\
&\le& \int_0^1 \|H(w+\xi z)-H(w)\|\, \|z\| d\xi+ H(w)z \\
&\le&  \|z\| \int_0^1 L \xi \|z\|d\xi + H(w)z = {L\over 2} \|z\|^2 + H(w)z .
\end{eqnarray*}\normalsize
We apply the preceding relation with $w=\l_k$ and $z=\a_k d_k$ and
obtain
$g(\l_k+\a_kd_k) -  g(\l_k) \le \a_k H(\l_k) d_k + {L\over 2} \a_k^2 \|d_k\|^2.$ 
By Eq.\ (\ref{approxNewdir}), we have $H(\l_k) d_k = - g(\l_k)+\e_k$.
Substituting this in the previous relation, this yields
$g(\l_k+\a_kd_k) \le (1-\a_k) g(\l_k) + \a_k \e_k + {L\over 2} \a_k^2
\|d_k\|^2.$ Moreover, using Assumption \ref{residual}(b), we have
$\|d_k\|^2 = \|H(\l_k)^{-1} (- g(\l_k) + \e_k)\|^2$
$\le \|H(\l_k)^{-1}\|^2\, \|- g(\l_k) + \e_k\|^2 \le  M^2 \Big(2 \| g(\l_k)\|^2 + 2 \|\e_k\|^2\Big)$.
Combining the above relations, we obtain
$\| g(\l_{k+1})\|\le(1-\a_k) \| g(\l_k)\| + M^2 L \a_k^2 \| g(\l_k)\|^2 +\a_k \|\e_k\|
+ M^2 L\a_k^2\|\e_k\|^2$\normalsize, establishing the desired relation.
\end{proof}

The proof of Proposition 2 follows from recursive application of the result in Lemma \ref{basicrel}. 

\begin{proof}{(Proposition \ref{mainresult})}
We show Eq.\ (\ref{quadconv}) using induction on the iteration
$m$. Using $\a_k=1$ in the statement of Lemma
\ref{basicrel}, we obtain
$\|g_{k+1}\| \le M^2 L \|g_k\|^2 +B\le {1/ 4M^2 L} + B$,
where the second inequality follows from the assumption
$\|g_k\|\le {1 / 2M^2 L}$. This establishes relation
(\ref{quadconv}) for $m=1$.
We next assume that (\ref{quadconv}) holds for some $m>0$, and show
that it also holds for $m+1$. Eq.\ (\ref{quadconv}) implies that
$\|g_{k+m}\|\le  {1/ 4M^2 L} + B+ {\delta / 4M^2L}.$
Using the assumption $B+M^2 L B^2 \le {\delta / 4 M^2 L}$, this
yields
$\|g_{k+m}\|\le  {1+2\delta / 4M^2 L} < {1/ 2M^2L},$
where the strict inequality follows from $\delta \in (0,1/2)$.
Using $\a_{k+m}=1$ in the generalized descent lemma \ref{basicrel}, we obtain
\[M^2L \|g_{k+m+1}\|\le \Big(M^2L \|g_{k+m}\|\Big)^2+B.\]\normalsize
Using Eq.\  (\ref{quadconv}), this implies that $M^2L \|g_{k+m+1}\|$
\begin{eqnarray*}
&\le& \left( {1\over 2^{2^m}} +
M^2 L B + \frac{\delta(2^{2^m-1}-1)} {2^{2^m}}\right)^2 + M^2L B\\
&=& {1\over 2^{2^{m+1}}} + {M^2LB \over 2^{2^m-1}} + \delta\,
\frac{2^{2^m-1}-1}{ 2^{2^{m+1}-1}}\\
&&+ M^2L \left(B+\frac{\delta}{M^2 L}\, \frac{(2^{2^m-1}-1)}{2^{2^m}}  \right)^2+ M^2L B.
\end{eqnarray*}
\normalsize
Using algebraic manipulations and the assumption $B+M^2 L B^2 \le
{\delta \over 4 M^2 L}$, this yields
\[\|g_{k+m+1}\|\le {1\over 2^{2^{m+1}}M^2 L} + B + \frac{\delta}{M^2 L}\,
\frac{(2^{2^{m+1}-1}-1)}{2^{2^{m+1}}},\]\normalsize completing the induction. Taking the
limit superior in Eq.\ (\ref{quadconv}) establishes the final
result.
\end{proof}

\subsection{Distributed backtracking line search}\label{sec_backtracking}

Proposition \ref{mainresult} establishes {\it local} quadratic convergence for properly selected members of the ADD family. To guarantee {\it global} convergence we modify ADD to use time varying step sizes $\a_{k}$ selected through distributed backtracking line search \cite[Algorithm 9.2]{boydbook}. Line search implementation requires computation of the gradient norm $\|g_{k}\| = \sum_{i=1}^{n} {g_{k}^{i}}^{2}$. This can be easily achieved using distributed consensus algorithms, e.g., \cite{ali}. However, since these consensus algorithms are iterative in nature, an approximate norm $\eta_{k}$ is computed in lieu of $\|g_{k}\|$. We assume that approximate gradient norms are computed with an error not exceeding a given constant $\gamma/2\geq0$,

\begin{equation}
	\Big|\eta_k - \|g_k\| \Big|\le \g /2,\label{errorstep}
\end{equation}
For fixed scalars $\sigma\in (0,1/2)$ and
$\beta\in (0,1)$, we set the stepsize $\a_k$ equal to
$\a_k=\beta^{m_k}$, where $m_k$ is the smallest nonnegative integer
that satisfies
\begin{equation}
	n_{k+1} \le (1-\sigma \beta^{m})\eta_k + B +\g.\label{inexactArmijo}
\end{equation}
The expression in (\ref{inexactArmijo}) coincides with the regular (centralized) backtracking line search except for the use of the approximate norm $\eta_k$ instead of the actual norm $\|g_{k}\|$ and the (small) additive constants $B$ and $\gamma$ respectively defined in (\ref{sumerr}) and (\ref{errorstep}).

While we introduce line search to ensure {\it global} convergence we start by showing that stepsizes selected according to the rule in (\ref{inexactArmijo}) do not affect {\it local} convergence. As we show next, this is because if $\|g_k\|\le 1/(2M^2 L)$ as required in Proposition \ref{mainresult} the rule in (\ref{inexactArmijo}) selects stepsizes $\a_{k}=1$.

\begin{proposition}\label{achoice}
If at iteration $k$ of ADD-$N$ the gradient norm satisfies $\|g_k\|\le 1/(2M^2 L)$, the inexact backtracking stepsize rule in (\ref{errorstep}) selects $\a_k=1$.
\end{proposition}
\begin{proof}{}
Replacing $\a_k=1$ in Lemma
\ref{basicrel} and using the definition of the constant $B$, we
obtain
$ \|g_{k+1}\| \le M^2L \|g_k\|^2 + B \le {1\over 2} \|g_k\| + B\le  (1-\sigma) \|g_k\| + B,$ \normalsize
where to get the last inequality, we used the fact that the constant
$\sigma$ used in the inexact backtracking stepsize rule satisfies
$\sigma\in (0,1/2)$. Using the condition on $\eta_k$ [cf.\ Eq.\
(\ref{errorstep})], this yields $n_{k+1}\le (1-\sigma) \eta_k + B + \g,$
showing that the steplength $\a_k=1$ satisfies condition
(\ref{inexactArmijo}) in the inexact backtracking stepsize rule.
\end{proof}

As per Proposition \ref{achoice}, if ADD-$N$ with backtracking line search is initialized at a point at which $\|g_0\|\leq 1/(2M^2 L)$, stepsizes $\a_{k}=1$ are used. Therefore, Proposition \ref{mainresult} holds, and convergence to the optimum $\l^{*}$ is quadratic, which in practice implies convergence in a few steps. Otherwise, selecting step sizes $\a_{k}$ satisfying (\ref{inexactArmijo}) ensures a strict decrease in the norm of the residual function as proven by the proposition below.

\begin{proposition}\label{globalresult}
Consider ADD-$N$ algorithms characterized by the iteration (\ref{update}) with $\bar H_{k}^{(N)}$ as defined in (\ref{dir}). Let Assumption \ref{residual} hold and stepsizes $\a_k$ being selected according to the inexact backtracking rule in (\ref{inexactArmijo}). Further assume that $\|g_k\|> {1/ 2 M^2 L}$ and that $N$ is chosen large enough to ensure that the constants $B$ and $\gamma$ in (\ref{sumerr}) and (\ref{errorstep}) satisfy
\begin{equation}
	B +2\g\le \frac{\beta}{16M^2 L},\label{sizerr}
\end{equation}
where $\beta$ is the backtracking rule constant and $M$ and $L$ are defined in Assumption \ref{residual}. Then, the gradient norm at iteration $k+1$ decreases by at least $\beta/(16 M^2 L)$,
\begin{equation}
	\|g_{k+1}\| \le \|g_k\| -\frac{\beta}{16 M^2 L}.
\end{equation}\end{proposition}

\begin{proof}{} For any $k\ge 0$, we define
$\bar{\a}_k = \frac{1}{2M^2 L (\eta_k +\g/2)}.$
In view of the condition on $\eta_k$ [cf.\ Eq.\ (\ref{errorstep})], we
have \begin{equation}\frac{1}{2M^2 L (\|g_k\| +\g)}\le \bar{\a}_k
\le \frac{1}{2M^2 L \|g_k\|}<1,\label{stepbd}\end{equation}\normalsize where
the last inequality follows by the assumption $\|g_k\|> {1\over 2
M^2 L}$. Using the preceding relation and substituting
$\a_k=\bar{\a}_k$ in the generalized descent lemma (lemma
\ref{basicrel}), we obtain:
\begin{eqnarray*}
\|g_{k+1}\|&\le &  \|g_k\|  + \bar{\a}_k \|\e_k\| + M^2 L\bar{\a}_k^2\|\e_k\|^2\\
&&-\bar{\a}_k\|g_k\|\Big(1- M^2 L \bar{\a}_k\|g_k\|\Big)\\
&\le & \|g_k\|  + \bar{\a}_k \|\e_k\| + M^2 L\bar{\a}_k^2\|\e_k\|^2\\
&&-\bar{\a}_k\|g_k\|\Big(1- M^2 L \frac{\|g_k\|}{2M^2 L\|g_k\|}\Big)\\
&\le &  \bar{\a}_k \e + M^2 L\bar{\a}_k^2\e^2+\Big(1-{\bar{\a}_k\over 2}\Big)\|g_k\|\\
&\le& B+\Big(1-{\bar{\a}_k\over 2}\Big)\|g_k\|,
\end{eqnarray*}\normalsize
where the second inequality follows from the definition of
$\bar{\a}_k$ and the third inequality follows by combining the facts
$\bar{\a}_k<1$, $\|\e_k\|\le \e$ for all $k$, and the definition of
$B$. The constant $\sigma$ used in the definition of the inexact
backtracking line search satisfies $\sigma\in (0,1/2)$, therefore,
it follows from the preceding relation that
$\|g_{k+1}\| \le  (1-\sigma \bar{\a}_k )\|g_k\| +B.\,$
Using condition (\ref{errorstep}) once again, this implies
$n_{k+1} \le (1-\sigma \bar{\a}_k )\eta_k +B +\g,\,$
showing that the steplength $\a_k$ selected by the inexact
backtracking line search satisfies $\a_k\ge \beta \bar{\a}_k$. From
condition (\ref{inexactArmijo}), we have $n_{k+1} \le (1-\sigma \a_k )\eta_k +B +\g,$ which implies
 $\|g_{k+1}\| \le  (1-\sigma \beta \bar{\a}_k )\|g_k\| +B+2\g.$
Combined with Eq.\ (\ref{stepbd}), this yields
\[\|g_{k+1}\| \le  \Big(1-\frac{\sigma \beta}{2M^2L (\|g_k\|+\g)}\Big)\|g_k\| +B+2\g.\]\normalsize
By \ref{sizerr}, we also have
$\g\le B+2\g \le  {\beta}/{16M^2 L},$
which in view of the assumption  $\|g_k\|> {1/ 2 M^2 L}$
implies that $\g\le \|g_k\|$. Substituting this  in the preceding
relation and using the fact $\a\in (0,1/2)$, we obtain
$\|g_{k+1}\| \le \|g_k\| -{\beta}/{8M^2L } +B+2\g.$
Combined with \ref{sizerr}, this yields the desired
result.
\end{proof}

Proposition \ref{globalresult} shows that if ADD-$N$ is initialized at a point with gradient norm $\|g_0\|> 1/2M^2L$ we obtain a decrease in the norm of the gradient of at least ${\beta}/{16M^2L }$. This holds true for all iterations as long as $\|g_k\|> 1/2M^2L$. This establishes that we need at most ${16\| g_0\|M^2 L}/{\beta}$ iterations until we obtain $\|g_k\|\le 1/2M^2L$. At this point the quadratic convergence result in Proposition \ref{mainresult} comes in effect and ADD-$N$ converges in a few extra steps.

\section{Consensus-based Newton}\label{cbm}

An alternative approach to obtain an approximation to the Newton step is to use consensus iterations. This consensus-based inexact Newton algorithm is pursued in \cite{cdc09} and shares a connection with the algorithm proposed here. While consensus-based Newton is developed for a primal dual formulation, it can be adapted to dual descent as pursued here. In consensus-based Newton, approximate Newton directions $d_k$ are found by solving the following consensus dynamic
\[d_k^{(i+1)} = (D_k+I)^{-1}(B_k+I)d_k^{(i)}-(D_k+I)^{-1}g_k,\]
where the splitting $H_k=(D_k+I)-(B_k+I)$ was used.

Observe that both consensus-based Newton and ADD use a choice of splitting.  ADD uses $D_k, B_k$ as opposed to $(D_k+I), (B_k+I)$ but the choice of splitting is not crucial.  With the appropriate choice of splitting, the consensus update becomes
\[ d^{(i+1)}_k =  D_k^{-1} B_k d_k^{(i)}- D_k^{-1}g_k.\]
Choosing the initial value $d_k^{(0)} = 0$ results in a sequence of approximations of the Newton direction as follows
\begin{align*}
d_k^{(1)} &=  D_k^{-1} B_k 0 - D_k^{-1}g_k= -D_k^{-1}g_k = -\bar H_k^{(0)}g_k\\
d_k^{(2)} &=  D_k^{-1} B_k (-D_k^{-1}g_k) - D_k^{-1}g_k = -\bar H_k^{(1)} g_k\\
&\vdots& \\
d_k^{(m)} &= - \sum_{i=0}^{m-1} {D}_k^{-\frac{1}{2}} ({D}_k^{-\frac{1}{2}} {B_k} {D}_k^{-\frac{1}{2}})^i {D}_k^{-\frac{1}{2}}  g_k = -\bar H_k^{(m-1)} g_k
\end{align*}
We observe that after $m$ consensus iterations our approximation $d_k^{(m)}$ is the same approximation arrived at by using ADD-N with $N= m-1$.  Therefore, ADD has the same behavior as a fixed iteration version of consensus-based Newton.

\begin{figure}
\includegraphics[width=\columnwidth]{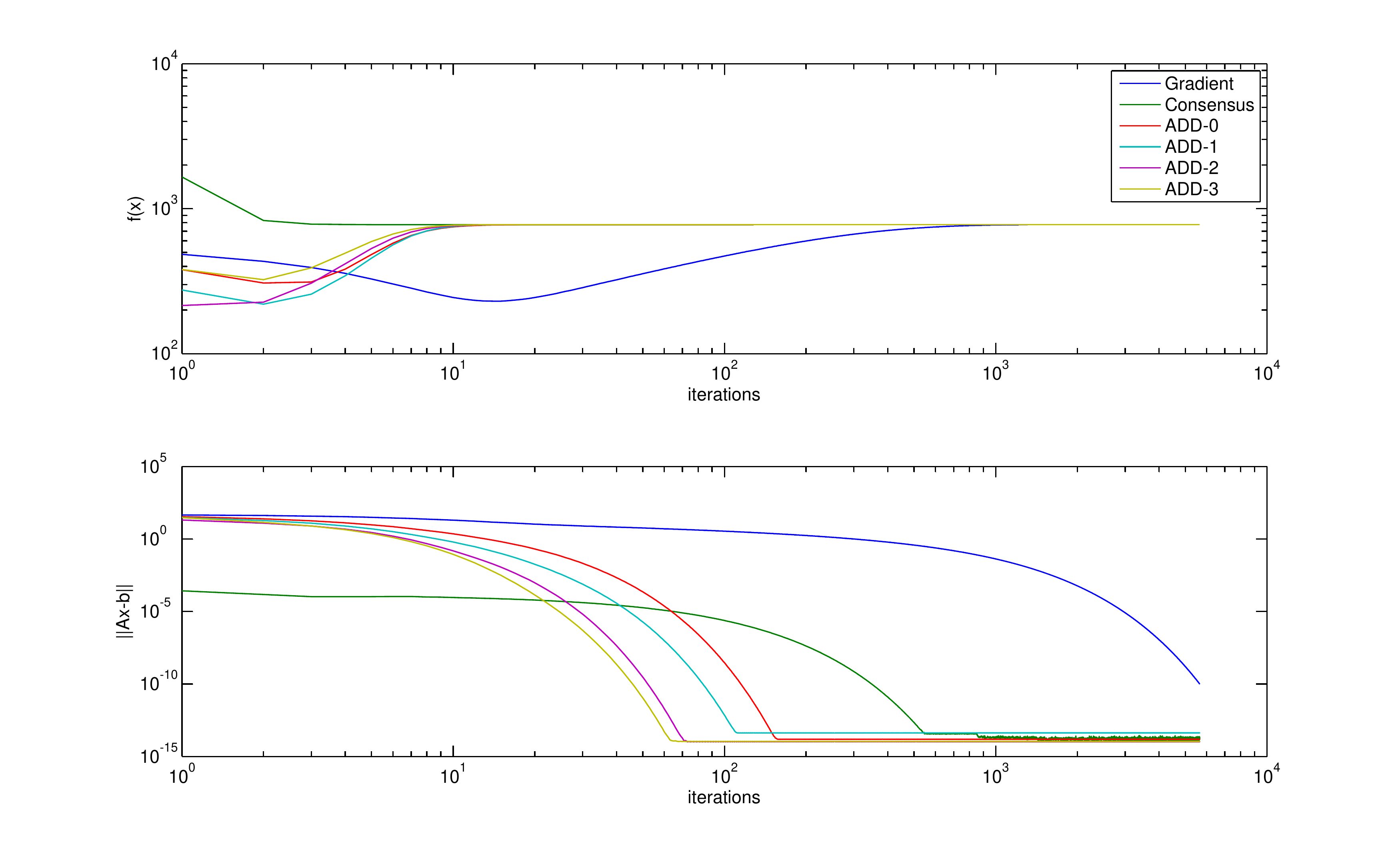}
\centering
\caption{\label{ex1}Primal objective (top), $f(x_k)$ and primal feasibility (bottom), $\|Ax_k-b\|$ with respect to dual descent iterations for a sample network optimization problem with 25 nodes and 75 edges.  ADD converges two orders of magnitude faster than gradient descent.  Increasing N reduces the number of iterations required to converge.}
\end{figure}

\begin{figure}
\includegraphics[width=\columnwidth]{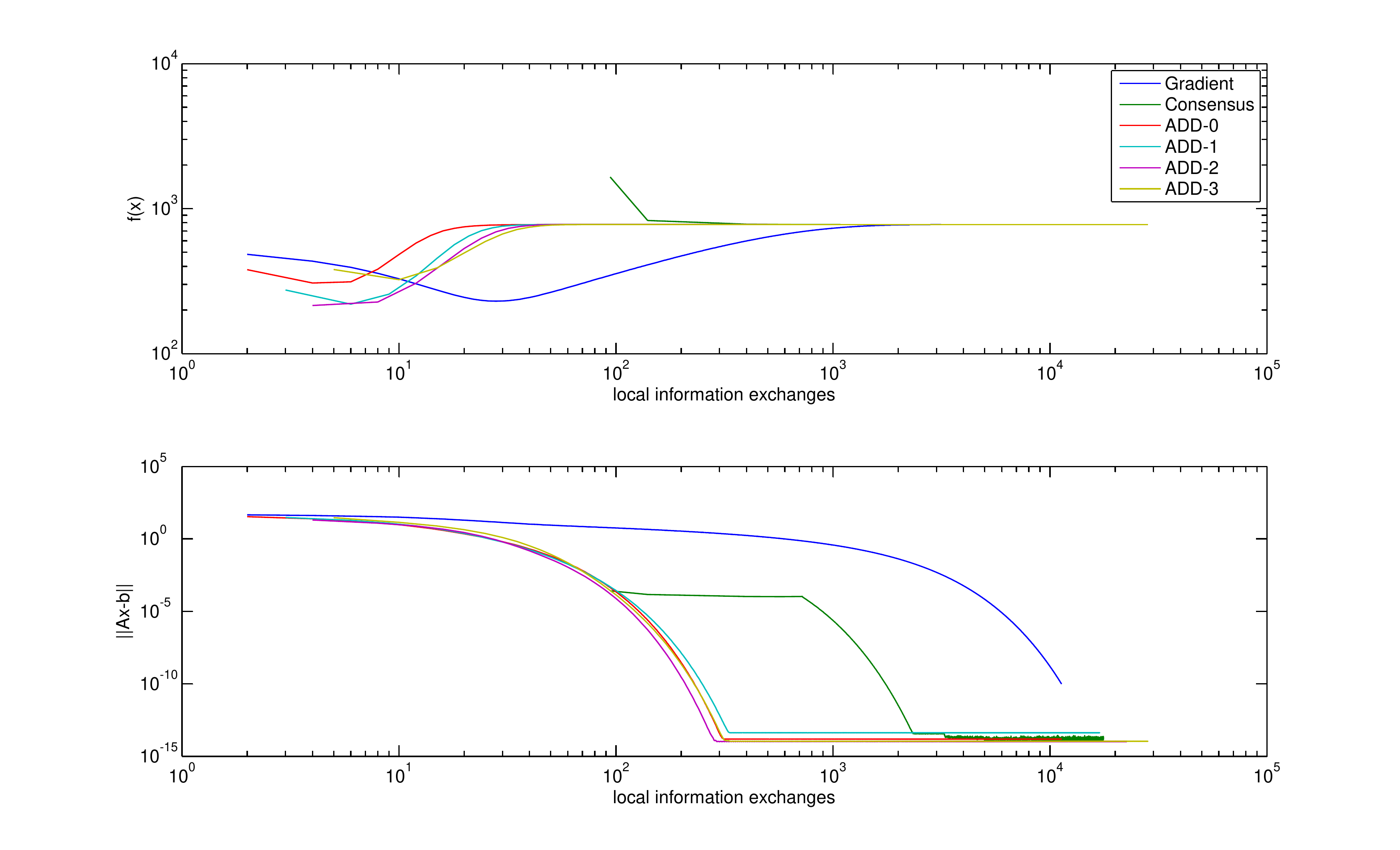}
\centering
\caption{\label{ex2}Primal objective (top), $f(x_k)$ and primal feasibility (bottom), $\|Ax_k-b\|$ with respect to number of local information exchanges for a sample network optimization problem with 25 nodes and 75 edges. ADD converges an order of magnitude faster than consensus-based Newton and two orders of magnitude faster than gradient descent.}
\end{figure}

\section{Numerical results}

Numerical experiments are undertaken to study ADD's performance with respect to the choice of the number of approximating terms $N$. These experiments show that $N=1$ or $N=2$ work best in practice. ADD is also compared to dual gradient descent \cite{ratejournal} and the consensus-based Newton method in \cite{cdc09} that we summarized in Section \ref{cbm}. These comparisons show that ADD convergence times are one to two orders of magnitude faster.

Figures \ref{ex1} and \ref{ex2} show convergence metrics for a randomly generated network with 25 nodes and 75 edges. Edges in the network are selected uniformly at random.  The flow vector $b$ is chosen to place sources and sinks a full diam$(\mathcal{G})$ away from each other. All figures show results for ADD-0 through ADD-3, gradient descent, and consensus-based Newton. In Fig. \ref{ex1}, objective value $f[x(\l_{k})]$ and constraint violation $\|Ax(\l_{k})-b\|$ are shown as functions of the iteration index $k$. As expected, the number of iterations required decreases with increasing $N$. The performance improvement, however, is minimal when increasing $N$ from 2 to 3. The convergence time of consensus-based Newton is comparable to ADD, while gradient descent is three orders of magnitude slower.


The comparison in Fig. \ref{ex1} is not completely accurate because different versions of ADD differ in the number of communication instances required per iteration. These numbers differ for consensus-based Newton and regular dual descent as well. Fig. \ref{ex2} normalizes the horizontal axis to demonstrate algorithms' progress with respect to the number of times nodes exchange information with their neighbors. Observe that in terms of this metric all versions of ADD are about an order of magnitude faster than consensus-based Newton and two orders orders of magnitude faster than gradient descent. The difference between Figs. \ref{ex1} and \ref{ex2} is due to the number of communication instances needed per each Newton iteration.


\begin{figure}
\includegraphics[width=\columnwidth]{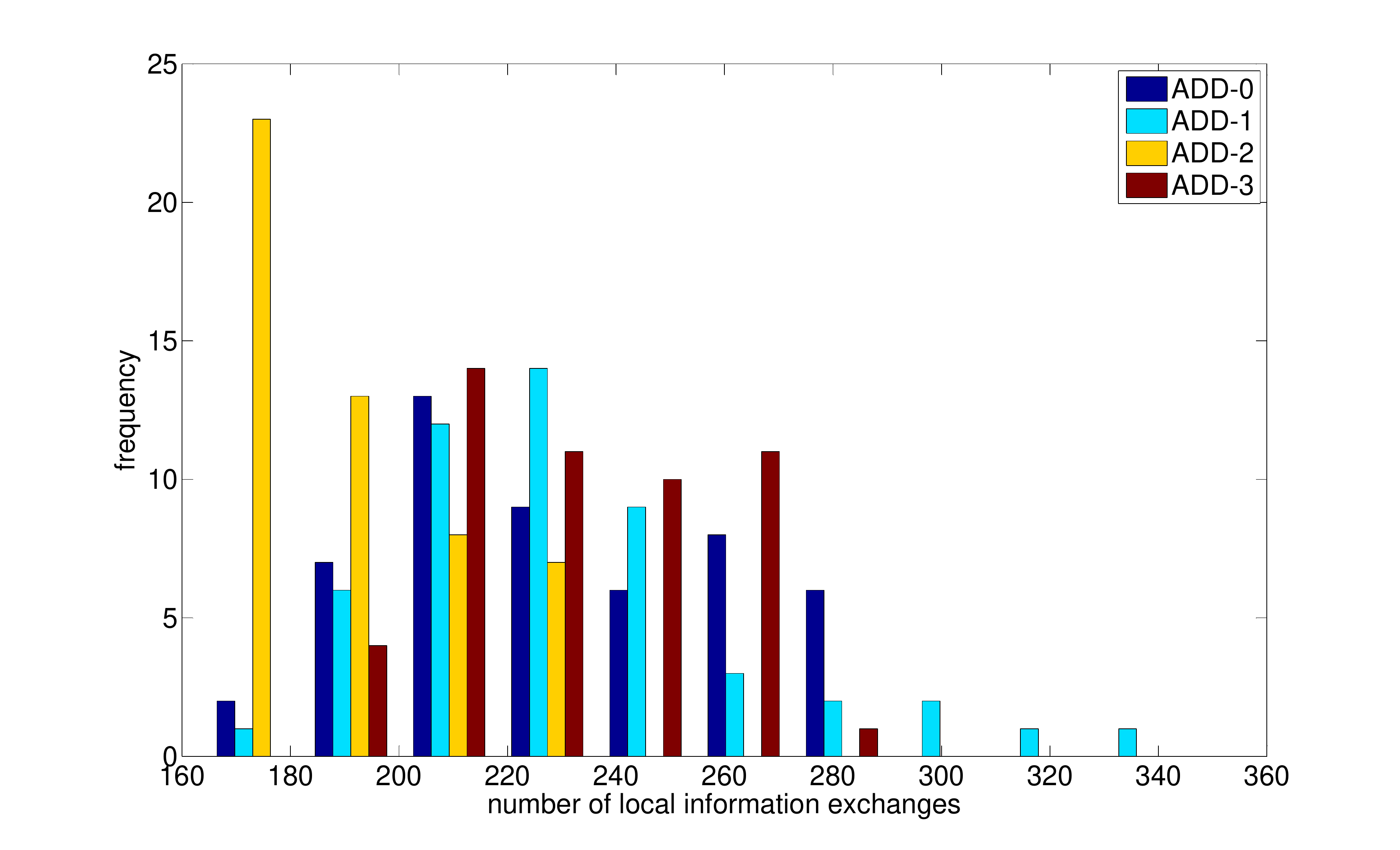}
\centering
\caption{\label{data1}Histogram of the number of local communications required to reach $\|g(\l_k)\|\le 10^{-10}$ for ADD-N with respect to parameter N, for 50 trials of the network optimization problem on random graphs with 25 nodes and 75 edges.  ADD-2 is shown to be the best on average by about $10\%$ indicating that with respect to communication cost, larger $N$ is not necessarily better.}
\end{figure}
Another important conclusion of Fig. \ref{ex2} is that even though increasing $N$ in ADD decreases the number of iterations required, there is not a strict decrease in the number of communications. Indeed, as can be appreciated in Fig. \ref{ex2},  ADD-2  requires fewer communications than ADD-3. This fact demonstrates an inherent trade off between spending communication instances to refine the Newton step $d_{k}^{(N)}$ versus using them to take a step. We further examine this phenomenon in Fig. \ref{data1}.  These experiments are on random graphs with 25 nodes and 75 edges chosen uniformly at random.  The flow vector $b$ is selected by placing a source and a sink at diam$(\mathcal{G})$ away from each other. We consider an algorithm to have converged when its residual $\|g_k\| \le 10^{-10}.$



The behavior of ADD is also explored for graphs of varying size and degree in Fig. \ref{data3}. As the graph size increases the performance gap between ADD and competing methods increases. Consistency of ADD is also apparent since the maximum, minimum, and average information exchanges required to solve (\ref{optnet}) for different network realizations are similar. This is not the case for neither consensus-based Newton nor gradient descent. Further note that ADD's communication cost increases only slightly with network size.

\begin{figure}
\includegraphics[width=\columnwidth]{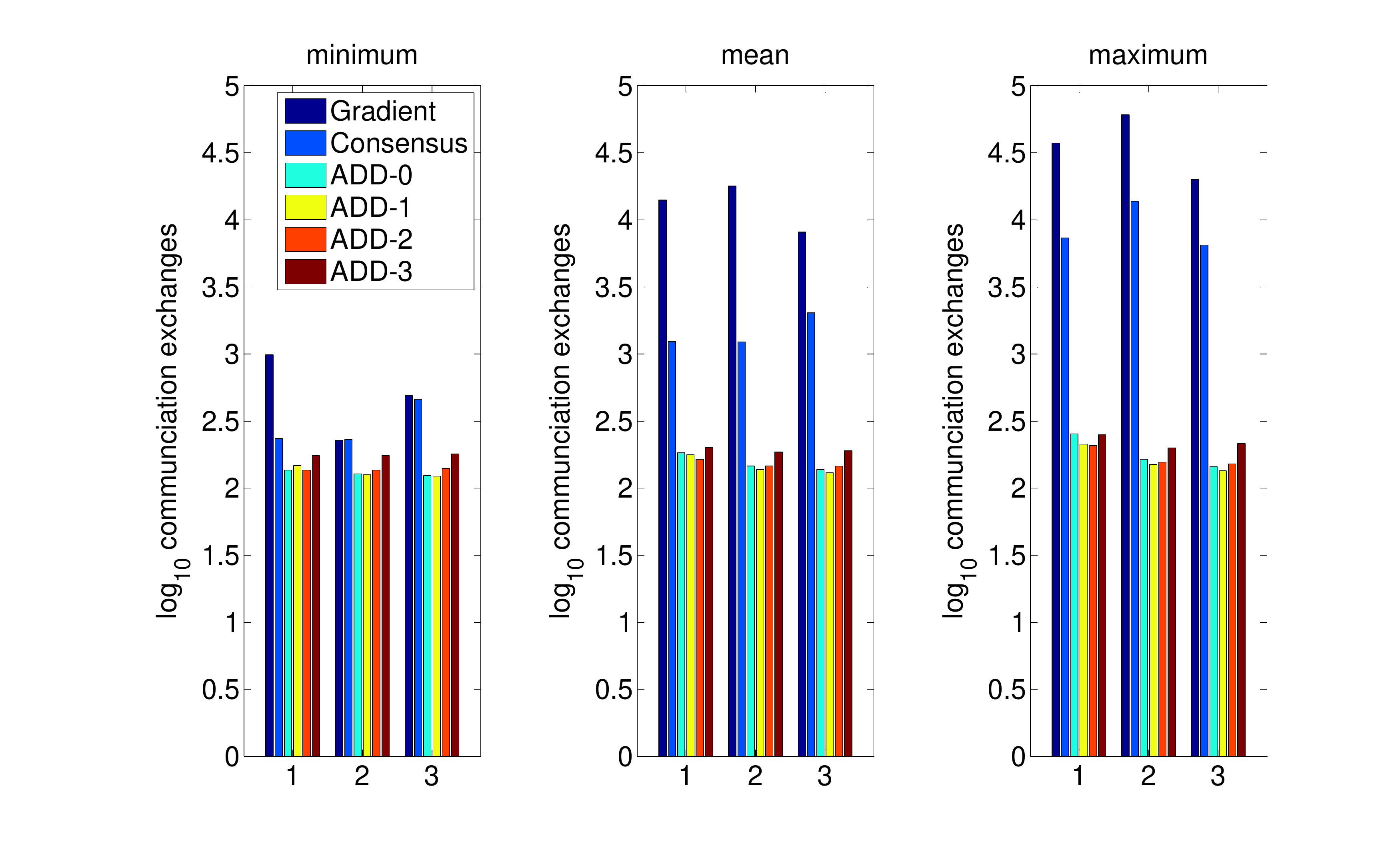}
\centering
\caption{\label{data3} Min (left), mean (center) and max (right) number of local communications required to reach $\|g(\l_k)\|\le 10^{-10}$ for gradient descent, consensus-based Newton and ADD, computed for 35 trials each on random graphs with 25 nodes and 75 edges(1), 50 nodes and 350 edges(2), and 100 nodes and 1000 edges (3).  The min and max are on the same order of magnitude for ADD, demonstrating small variance.  }
\end{figure}

\section{Conclusion}

A family of accelerated dual descent (ADD) algorithms to find optimal network flows in a distributed manner was introduced. Members of this family are characterized by a single parameter $N$ determining the accuracy in the approximation of the dual Newton step. This same parameter controls the communication cost of individual algorithm iterations. We proved that it is always possible to find members of this family for which convergence to optimal operating points is quadratic.

Simulations demonstrated that $N=1$ and $N=2$, respectively denoted as ADD-1 and ADD-2 perform best in practice. ADD-1 corresponds to Newton step approximations using information from neighboring nodes only, while ADD-2 requires information from nodes two hops away. ADD-1 and ADD-2 outperform gradient descent by two orders of magnitude and a related consensus-based Newton method by one order of magnitude.

Possible extensions include applications to network utility maximization \cite{wei}, general wireless communication problems \cite{wireless}, and stochastic settings \cite{eso}.

\bibliographystyle{amsplain}
\bibliography{distributed}
\end{document}